\documentclass[reqno, 12pt]{amsart}
\usepackage{amsaddr}
\usepackage[latin1]{inputenc}
\usepackage[T1]{fontenc}
\usepackage[english]{babel}
\usepackage[babel]{microtype}
\usepackage{amsmath,amssymb,amsthm,amsrefs}
\usepackage{mathrsfs,mathtools}
\usepackage[a4paper, margin=1.2in]{geometry}
\usepackage{commath}
\usepackage{bbm}
\usepackage{dsfont}
\usepackage{lmodern}
\usepackage{cases}
\usepackage[shortlabels]{enumitem}
\usepackage{graphicx}
\usepackage{url}
\usepackage{float}
\usepackage{xcolor}
\usepackage{hyperref}
\usepackage{orcidlink}

\numberwithin{equation}{section}
\newtheorem{thm}    {Theorem}
\newtheorem{lem}      [thm] {Lemma}
\newtheorem{con} [thm] {Conjecture}

\newtheorem{exm}    [thm] {Example}

\setlist[enumerate,1]{label=(\roman*)}
\DeclareMathOperator{\N}{\mathbb{N}}

\DeclareMathOperator{\Q}{\mathbb{Q}}

\begin{document}

\title{On Means of Near Continued Fractions}

\author[T. J. ~Spier]{Thom\'as Jung Spier\,\orcidlink{0000-0003-4822-3046}}
\address{Department of Computer Science, Federal University of Minas Gerais, Belo Horizonte, Minas Gerais, Brazil}
\email{thomasjung@dcc.ufmg.br}

%%%%%%%%%%%%%%%%%%%%%%%%%%%%%%%%%%%%%%%%%%%%%%%%%%%%%%%%%%%%%%%%%%%%%%%%%%%%%%%%

\begin{abstract}
In this work, we present continued fractions for the arithmetic, geometric, harmonic and cotangent means of $[a_0,a_1,\dots,a_k]$ and $[a_0,a_1,\dots,a_k,a_{k+1}]$, and some of their applications.
\end{abstract}

%%%%%%%%%%%%%%%%%%%%%%%%%%%%%%%%%%%%%%%%%%%%%%%%%%%%%%%%%%%%%%%%%%%%%%%%%%%%%%%%

\makeatletter
\@namedef{subjclassname@2020}{\textup{2020} Mathematics Subject Classification}
\makeatother

%%%%%%%%%%%%%%%%%%%%%%%%%%%%%%%%%%%%%%%%%%%%%%%%%%%%%%%%%%%%%%%%%%%%%%%%%%%%%%%%

\clearpage\maketitle
\thispagestyle{empty}

\section{Introduction}\label{sec.intro}

In this work continued fractions will be denoted by,

\[[a_0,a_1,a_2,\dots,a_j,\dots]:=a_0+\dfrac{1}{a_1+\dfrac{1}{a_2+\dfrac{1}{\dots+\dfrac{1}{a_j+\dots}}}}.\] 

Two finite continued fractions are said to be {\it near} if one of them has one more term than the other. A pair of near continued fractions is represented by $[a_0,a_1,\dots,a_k]$ and $[a_0,a_1,\dots,a_k,a_{k+1}]$. In particular, two consecutive convergents of an irrational number are near continued fractions.

Our main theorem gives formulas for the \textit{arithmetic}, \textit{geometric}, \textit{harmonic} and \textit{cotangent means} of near continued fractions. This theorem must certainly be known, but unfortunately we do not find it in the literature and we do not know who was the first to prove it.

\begin{thm}\label{gmcf}(Means of near continued fractions) Let $a_0,a_1,\dots,a_{k+1}$ be positive real numbers. Then,
	
\[\dfrac{[a_0,a_1,\dots,a_k]+ [a_0,a_1,\dots,a_k,a_{k+1}]}{2}=[a_0,a_1,\dots,a_k,2a_{k+1},a_k,\dots,a_1],\]
	
\[\sqrt{[a_0,a_1,\dots,a_k]\cdot [a_0,a_1,\dots,a_k,a_{k+1}]}=[a_0,\overline{a_1,\dots,a_k,2a_{k+1},a_k,\dots,a_1,2a_0}],\]

\[\dfrac{2}{\dfrac{1}{[a_0,a_1,\dots,a_k]}+\dfrac{1}{ [a_0,a_1,\dots,a_k,a_{k+1}]}}=[a_0,a_1,\dots,a_k,2a_{k+1},a_k,\dots,a_1,a_0],\]

\[\cot\left(\dfrac{\cot^{-1}[a_0,a_1,\dots,a_k]+ \cot^{-1}[a_0,a_1,\dots,a_k,a_{k+1}]}{2}\right)=\]

\[=[\overline{a_0,a_1,\dots,a_k,2a_{k+1},a_k,\dots,a_1,a_0}].\]

\end{thm}

Theorem~\ref{gmcf} also admits a complex analogue.

\begin{thm}\label{gmcfcomplex}(Means of near complex continued fractions) Let $a_0,a_1,\dots,a_{k+1}$ be positive real numbers. Then,
	
	\[\dfrac{[a_0,a_1,\dots,a_k,a_{k+1}-i]+ [a_0,a_1,\dots,a_k,a_{k+1}+i]}{2}=\]
	
	\[=[a_0,a_1,\dots,a_k,a_{k+1},a_{k+1},a_k,\dots,a_1],\]
	
	\[\sqrt{[a_0,a_1,\dots,a_k,a_{k+1}-i]\cdot [a_0,a_1,\dots,a_k,a_{k+1}+i]}=\]
	
	\[=[a_0,\overline{a_1,\dots,a_k,a_{k+1},a_{k+1},a_k,\dots,a_1,2a_0}],\]
	
	\[\dfrac{2}{\dfrac{1}{[a_0,a_1,\dots,a_k,a_{k+1}-i]}+\dfrac{1}{ [a_0,a_1,\dots,a_k,a_{k+1}+i]}}=\]
	
	\[=[a_0,a_1,\dots,a_k,a_{k+1},a_{k+1},a_k,\dots,a_1,a_0],\]
	
	\[\cot\left(\dfrac{\cot^{-1}[a_0,a_1,\dots,a_k,a_{k+1}-i]+ \cot^{-1}[a_0,a_1,\dots,a_k,a_{k+1}+i]}{2}\right)=\]
	
	\[=[\overline{a_0,a_1,\dots,a_k,a_{k+1},a_{k+1},a_k,\dots,a_1,a_0}].\]
	
\end{thm}

The Theorems~\ref{gmcf} and~\ref{gmcfcomplex} are proved using only algebraic manipulations of continuant polynomials, which are numerators and denominators of continued fractions. We note that the positive real number condition in both statements is just to ensure convergence and that the objects are well defined, but the results apply to other types of continued fraction entries and notions of convergence.

The \textit{geometric mean} formula in Theorem~\ref{gmcf} is essentially contained in an article by Van der Poorten and Walsh~\cite[p. 52, Thm. 1]{van1999note}, but it is written completely differently. The \textit{arithmetic mean} formula in Theorem~\ref{gmcf} should also be compared with a result that appears in Mend\`es France work~\cite{france1973fractions} that was further explored by Van der Poorten and Shallit~\cite[p. 239, Prop. 2]{van1992folded} and~\cite[p. 604, Prop. 3]{van2002symmetry}. The result in these three articles is analogous to the \textit{arithmetic mean} formula in Theorem~\ref{gmcf}, but with some changed signs, and is used to transform some series into continued fractions.

In Section~\ref{sec.continuant.pol} we define the continuant polynomials, present some of their properties, and use them to prove Theorems~\ref{gmcf} and~\ref{gmcfcomplex}. In Sections~\ref{pell-factor} and~\ref{mordell}, we reinterpret some classical results about Pell's equation in terms of Theorems~\ref{gmcf} and~\ref{gmcfcomplex}.

%%%%%%%%%%%%%%%%%%%%%%%%%%%%%%%%%%%%%%%%%%%%%%%%%%%%%%%%%%%%%%%%%%%%%%%%%%%%%%%%

\section{Continuant Polynomials}\label{sec.continuant.pol}

We begin by recalling Euler's \textit{continuant} polynomials \cite{euler1744fractionibus} (or see~\cite[p. 287]{graham1989concrete}). Define recursively, 

\[K[\,\,]=1,\quad K[a_0]=a_0,\quad K[a_0,a_1]=a_0a_1+1,\]

\noindent and, $K[a_0,\dots,a_n]=a_n\cdot K[a_0,\dots,a_{n-1}]+K[a_0,\dots,a_{n-2}], \,\forall n\geq 2$. Note that $K[a_0,\dots,a_n]$ is a multilinear polynomial in the variables $a_0$, $a_1$, $\dots$, $a_n$.

The \textit{continuant} polynomials are the numerators and denominators of continued fractions:

\[[a_0,a_1,\dots,a_n]=\dfrac{K[a_0,a_1,\dots,a_n]}{K[a_1,\dots,a_n]},\,\quad\forall n\in \N.\]

The next two lemmas present the main properties of continuants that will be needed later.

\begin{lem}[Properties of continuants]\label{propcont} Let $n$ and $j$ be two natural numbers with $j<n$. Then,
\begin{enumerate}[label=(\alph*)]
	\item $K[a_0,a_1,\dots,a_{n-1},a_n]=K[a_n,a_{n-1},\dots,a_1,a_0]$;
	
	\item $K[a_0,\dots,a_{j-1},a_j,a_{j+1},\dots,a_n]=a_j\cdot K[a_0,\dots,a_{j-1}]\cdot K[a_{j+1},\dots,a_n]$
	
	$+K[a_0,\dots,a_{j-2}]\cdot K[a_{j+1},\dots,a_n]+K[a_0,\dots,a_{j-1}]\cdot K[a_{j+2},\dots,a_n]$;
	
	\item $K[a_0,\dots,a_j,a_{j+1},\dots,a_n]=$
		
	$=K[a_0,\dots,a_j]\cdot K[a_{j+1},\dots,a_n]+K[a_0,\dots,a_{j-1}]\cdot K[a_{j+2},\dots,a_n]$;
	
	\item $K[a_0,\dots,a_n]\cdot K[a_1,\dots,a_{n+1}]-K[a_1,\dots,a_n]\cdot K[a_0,\dots,a_{n+1}]=(-1)^{n+1}$.
\end{enumerate}
\end{lem}
\begin{proof} All items can be proved by induction on the number of entries in the continuant polynomials. For reference, see~\cite[p. 289, Eq. 6.131--6.134]{graham1989concrete}.
\end{proof}

\begin{lem}[Properties of symmetric continuants]\label{propcontsym} Let $k$ be a natural number. Then,
\begin{enumerate}[label=(\alph*)]
	\item $K[a_0,\dots,a_k,2a_{k+1},a_k,\dots,a_0]=2\cdot K[a_0,\dots,a_k]\cdot K[a_0,\dots,a_k,a_{k+1}];$
	\item $K[a_0,a_1,\dots,a_k,2a_{k+1},a_k,\dots,a_1]=$
		
	$=K[a_0,\dots,a_k]\cdot K[a_1,\dots,a_{k+1}]+K[a_0,\dots,a_{k+1}]\cdot K[a_1,\dots,a_k];$
	\item $K[a_0,\dots,a_{k+1},a_{k+1},\dots,a_0]=K[a_0,\dots,a_{k+1}]^2+K[a_0,\dots,a_k]^2=$
		
	$=K[a_0,\dots,a_k,a_{k+1}-i]\cdot K[a_0,\dots,a_k,a_{k+1}+i];$
	\item $K[a_0,a_1,\dots,a_{k+1},a_{k+1},\dots,a_1]=$
		
	$=K[a_0,\dots,a_{k+1}]\cdot K[a_1,\dots,a_{k+1}]+K[a_0,\dots,a_k]\cdot K[a_1,\dots,a_k].$
\end{enumerate}
\end{lem} 
\begin{proof} The proof uses the properties of continuants presented in Lemma~\ref{propcont}.
	
$(a)$ $K[a_0,\dots,a_k,2a_{k+1},a_k,\dots,a_0]=$

$=2a_{k+1}\cdot K[a_0,\dots,a_k]^2+2\cdot K[a_0,\dots,a_{k-1}]\cdot K[a_0,\dots,a_k]=$

$=2\cdot K[a_0,\dots,a_k]\cdot (a_{k+1}\cdot K[a_0,\dots,a_k]+K[a_0,\dots,a_{k-1}])=$

$=2\cdot K[a_0,\dots,a_k]\cdot K[a_0,\dots,a_n,a_{k+1}],$

where the first equality is due to items $(a)$ and $(b)$ of Lemma~\ref{propcont}.

$(b)$ $K[a_0,\dots,a_k,2a_{k+1},a_k,\dots,a_1]=$

$=2a_{k+1}\cdot K[a_0,\dots,a_k]\cdot K[a_1,\dots,a_k]+K[a_0,\dots,a_{k-1}]\cdot K[a_1,\dots,a_k]$

$+K[a_0,\dots,a_k]\cdot K[a_1,\dots,a_{k-1}]=$

$=K[a_0,\dots,a_k]\cdot (a_{k+1}\cdot K[a_1,\dots,a_k]+K[a_1,\dots,a_{k-1}])+$

$K[a_1,\dots,a_k]\cdot (a_{k+1}\cdot K[a_0,\dots,a_k]+K[a_0,\dots,a_{k-1}])=$

$=K[a_0,\dots,a_k]\cdot K[a_1,\dots,a_{k+1}]+K[a_0,\dots,a_{k+1}]\cdot K[a_1,\dots,a_k],$

where the first equality is due to items $(a)$ and $(b)$ of Lemma~\ref{propcont}.

The proofs of items $(c)$ and $(d)$ are analogous and follow from items $(a)$ and $(c)$ of Lemma~\ref{propcont}.
\end{proof}

Theorems~\ref{gmcf} and~\ref{gmcfcomplex} are immediate consequences of Lemmas~\ref{propcont} and~\ref{propcontsym}. 

\begin{proof}[Proof of Theorems~\ref{gmcf} and~\ref{gmcfcomplex}.] We only prove Theorem~\ref{gmcf}, the proof of Theorem~\ref{gmcfcomplex} is analogous and uses items $(c)$ and $(d)$ from Lemma~\ref{propcontsym}.
	
\textbf{Arithmetic mean:} 

\[\dfrac{[a_0,a_1,\dots,a_k]+ [a_0,a_1,\dots,a_k,a_{k+1}]}{2}=\dfrac{\dfrac{K[a_0,a_1,\dots,a_k]}{K[a_1,\dots,a_k]}+ \dfrac{K[a_0,a_1,\dots,a_{k+1}]}{K[a_1,\dots,a_{k+1}]}}{2}=\]

\[=\dfrac{K[a_0,a_1,\dots,a_k]\cdot K[a_1,\dots,a_{k+1}]+K[a_0,a_1,\dots,a_{k+1}]\cdot K[a_1,\dots,a_k]}{2\cdot K[a_1,\dots,a_k]\cdot K[a_1,\dots,a_{k+1}]}=\]

\[=\dfrac{K[a_0,a_1,\dots,a_k,2a_{k+1},a_k,\dots,a_1]}{K[a_1,\dots,a_k,2a_{k+1},a_k,\dots,a_1]}=[a_0,a_1,\dots,a_k,2a_{k+1},a_k,\dots,a_1],\]

\noindent where the third equality is by items $(a)$ and $(b)$ of Lemma~\ref{propcontsym}.

\textbf{Geometric mean:}

Consider $x=[a_0,\overline{a_1,\dots,a_k,2a_{k+1},a_k,\dots,a_1,2a_0}]$, we need to prove that $x$ is equal to $\sqrt{[a_0,a_1,\dots,a_k]\cdot [a_0,a_1,\dots,a_k,a_{k+1}]}$. Note that,

\[x=[a_0,a_1,\dots,a_k,2a_{k+1},a_k,\dots,a_1,a_0+x]=\]

\[=\dfrac{K[a_0,a_1,\dots,a_k,2a_{k+1},a_k,\dots,a_1,a_0+x]}{K[a_1,\dots,a_k,2a_{k+1},a_k,\dots,a_1,a_0+x]}=\]

\[=\dfrac{K[a_0,a_1,\dots,a_k,2a_{k+1},a_k,\dots,a_1]\cdot x+K[a_0,a_1,\dots,a_k,2a_{k+1},a_k,\dots,a_1,a_0]}{K[a_1,\dots,a_k,2a_{k+1},a_k,\dots,a_1]\cdot x+K[a_1,\dots,a_k,2a_{k+1},a_k,\dots,a_1,a_0]},\]

\noindent and by item $(a)$ of Lemma~\ref{propcont} and rearranging the terms,

\[K[a_1,\dots,a_k,2a_{k+1},a_k,\dots,a_1]x^2=K[a_0,a_1,\dots,a_k,2a_{k+1},a_k,\dots,a_1,a_0]\iff\]

\[x=\sqrt{\dfrac{K[a_0,a_1,\dots,a_k,2a_{k+1},a_k,\dots,a_1,a_0]}{K[a_1,\dots,a_k,2a_{k+1},a_k,\dots,a_1]}}=\]
	
\[=\sqrt{\dfrac{2\cdot K[a_0,a_1\dots,a_k]\cdot K[a_0,a_1,\dots,a_k,a_{k+1}]}{2\cdot K[a_1,\dots,a_k] \cdot K[a_1,\dots,a_k,a_{k+1}]}}=\]

\[=\sqrt{[a_0,a_1,\dots,a_k]\cdot [a_0,a_1,\dots,a_k,a_{k+1}]},\]

\noindent where the second equality is by item $(a)$ of Lemma~\ref{propcontsym}. 

\textbf{Harmonic mean:} Simply notice that,

\[\dfrac{2}{\dfrac{1}{[a_0,a_1,\dots,a_k]}+\dfrac{1}{ [a_0,a_1,\dots,a_k,a_{k+1}]}}=[a_0,a_1,\dots,a_k,2a_{k+1},a_k,\dots,a_1,a_0]\iff\] 

\[\dfrac{[0,a_0,a_1,\dots,a_k]+ [0,a_0,a_1,\dots,a_k,a_{k+1}]}{2}=[0,a_0,a_1,\dots,a_k,2a_{k+1},a_k,\dots,a_1,a_0],\]

\noindent which is true by the formula for the \textit{arithmetic mean}.

\textbf{Cotangent mean:}

Consider $x:=[\overline{a_0,a_1,\dots,a_k,2a_{k+1},a_k,\dots,a_1,a_0}]$, we need to prove that $x$ is equal to $\cot\left(\dfrac{\cot^{-1}[a_0,a_1,\dots,a_k]+\cot^{-1}[a_0,a_1,\dots,a_k,a_{k+1}]}{2}\right)$. Note that,

\[x=[a_0,a_1,\dots,a_k,2a_{k+1},a_k,\dots,a_1,a_0,x]=\]

\[=\dfrac{K[a_0,a_1,\dots,a_k,2a_{k+1},a_k,\dots,a_1,a_0,x]}{K[a_1,\dots,a_k,2a_{k+1},a_k,\dots,a_1,a_0,x]}=\]

\[=\dfrac{K[a_0,a_1,\dots,a_k,2a_{k+1},a_k,\dots,a_1,a_0]\cdot x+K[a_0,a_1,\dots,a_k,2a_{k+1},a_k,\dots,a_1]}{K[a_1,\dots,a_k,2a_{k+1},a_k,\dots,a_1,a_0]\cdot x+K[a_1,\dots,a_k,2a_{k+1},a_k,\dots,a_1]},\]

\noindent and by item $(a)$ of Lemma~\ref{propcont} and rearranging the terms,

\[\dfrac{x^2-1}{2x}=\dfrac{K[a_0,a_1,\dots,a_k,2a_{k+1},a_k,\dots,a_1,a_0]-K[a_1,\dots,a_k,2a_{k+1},a_k,\dots,a_1]}{2\cdot K[a_1,\dots,a_k,2a_{k+1},a_k,\dots,a_1,a_0]}=\]

\[=\dfrac{K[a_0,a_1,\dots,a_k]\cdot K[a_0,a_1,\dots,a_k,a_{k+1}]-K[a_1,\dots,a_k]\cdot K[a_1,\dots,a_k,a_{k+1}]}{K[a_0,a_1,\dots,a_k]\cdot K[a_1,\dots,a_k,a_{k+1}]+K[a_0,a_1,\dots,a_k,a_{k+1}]\cdot K[a_1,\dots,a_k]}=\]

\[=\dfrac{[a_0,a_1,\dots,a_k]\cdot [a_0,a_1,\dots,a_k,a_{k+1}]-1}{[a_0,a_1,\dots,a_k]+[a_0,a_1,\dots,a_k,a_{k+1}]},\]

\noindent where the second equality is due to items $(a)$ and $(b)$ of Lemma~\ref{propcontsym}. As a consequence,

\[\dfrac{x^2-1}{2x}=\dfrac{[a_0,a_1,\dots,a_k]\cdot [a_0,a_1,\dots,a_k,a_{k+1}]-1}{[a_0,a_1,\dots,a_k]+[a_0,a_1,\dots,a_k,a_{k+1}]}.\]

Using this last equation and the formula for the cotangent of the sum,

\[\cot(y+z)=\dfrac{\cot y\cdot \cot z-1}{\cot y+\cot z},\]

\noindent we obtain $2\cot^{-1}x=\cot^{-1}[a_0,a_1,\dots,a_k]+\cot^{-1}[a_0,a_1,\dots,a_k,a_{k+1}]$, from which the result follows.
\end{proof}

\section{Pell Equation and Factorization}\label{pell-factor}\

For a rational number $\dfrac{p}{q}>1$, which is not a square in $\Q$, it is well known that,

\[\sqrt{\dfrac{p}{q}}=[a_0/2,\overline{a_1,a_2,\dots,a_2,a_1,a_0}],\quad\quad \dfrac{1+\sqrt{\dfrac{p}{q}}}{2}=[(1+b_0)/2,\overline{b_1,b_2,\dots,b_2,b_1,b_0}],\]

\noindent where $a_j$ and $b_j$ are natural numbers for each $j$, $a_0$ is even and $b_0$ is odd, and with the central words, $(a_1,a_2,\dots,a_2,a_1)$ and $(b_1,b_2,\dots,b_2,b_1)$, palindromes. Assume from now on that the periods of these continued fractions are minimal and equal to $l$ and $m$, respectively. 

For the natural numbers, these continued fraction expansions are of particular interest, as they give rise to the fundamental solutions of Pell's equation. If $\dfrac{p}{q}$ is equal to a natural number $n$, then it is well known that,

\[K[a_0/2,a_1,a_2,\dots,a_2,a_1]^2-n\cdot K[a_1,a_2,\dots,a_2,a_1]^2=(-1)^l,\]

\[(2\cdot K[b_0/2,b_1,b_2,\dots,b_2,b_1])^2-n\cdot K[b_1,b_2,\dots,b_2,b_1]^2=4(-1)^m.\] 

We will focus on the first equation above, but similar considerations hold for the second equation as well. Recall that the period $l$ is odd if and only if the negative Pell equation $x^2-ny^2=-1$ has an integer solution. As is well known, if $n$ is divisible by a prime $p$ congruent to $3$ modulo $4$, then $x^2-ny^2=-1$ has no integer solutions because $-1$ is not a square modulo $p$. It follows that the period $l$ is even whenever $n$ is divisible by a prime congruent to $3$ modulo $4$.

Note that if $l=2k+2$ is even, then,

\[\sqrt{n}=[a_0/2,\overline{a_1,a_2,\dots,a_k,a_{k+1},a_k,\dots,a_2,a_1,a_0}].\]

\noindent whereas if $l=2k+1$ is odd, then,

\[\sqrt{n}=[a_0/2,\overline{a_1,a_2,\dots,a_k,a_k,\dots,a_2,a_1,a_0}].\]

In this first case, we can apply the \textit{geometric mean} formula in Theorem~\ref{gmcf} to obtain $\sqrt{n}=\sqrt{[a_0/2,a_1,\dots,a_k]\cdot [a_0/2,a_1,\dots,a_k,a_{k+1}/2]}$, which is equivalent to,

\[n=[a_0/2,a_1,\dots,a_k]\cdot [a_0/2,a_1,\dots,a_k,a_{k+1}/2].\]

We conclude that the natural number $n$ can be written as the product of near continued fractions. As we will now see, we can obtain a non-trivial factorization of $n$ from this product whenever $a_{k+1}$ is even. Observe that if $a_{k+1}$ is even, then $n$ is the product of two near continued fractions with natural numbers as entries.

Since $K[a_0/2,a_1,\dots,a_k]$ and $K[a_1,\dots,a_k]$, and $K[a_0/2,a_1,\dots,a_k,a_{k+1}/2]$ and $K[a_1,\dots,a_k,a_{k+1}/2]$ are pairs of coprime natural numbers, it follows that $n$ can be factored as the product of two natural numbers as,

\[n=\dfrac{K[a_0/2,a_1,\dots,a_k]}{K[a_1,\dots,a_k,a_{k+1}/2]}\cdot \dfrac{K[a_0/2,a_1,\dots,a_k,a_{k+1}/2]}{K[a_1,\dots,a_k]}.\]

The two factors of the natural number $n$ obtained this way are also coprime because $K[a_0/2,a_1,\dots,a_k]$ and $K[a_0/2,a_1,\dots,a_k,a_{k+1}/2]$ are coprime.

Furthermore, the factorization of $n$ obtained in this way is non-trivial. Indeed, if this were not the case, then, $K[a_0/2,a_1,\dots,a_k]=K[a_1,\dots,a_k,a_{k+1}/2]$, and from item $(d)$ in Lemma~\ref{propcont}, $K[a_0/2,a_1,\dots,a_k]\cdot K[a_1,\dots,a_k,a_{k+1}/2]-K[a_1,\dots,a_k]\cdot K[a_0/2,a_1,\dots,a_k,a_{k+1}/2]=(-1)^{k+1}$. As a consequence, $K[a_0/2,a_1,\dots,a_k]^2-n\cdot K[a_1,\dots,a_k]^2=(-1)^{k+1}$ is a nontrivial solution for the Pell equation of $n$ that is smaller than the fundamental solution, which is impossible. 

In conclusion, from the continued fraction of $\sqrt{n}$ it is sometimes possible to obtain a non-trivial factorization of $n$ as the product of two coprime natural numbers. In fact, as we have shown, this is possible whenever the period and central term of the continued fraction $\sqrt{n}$ are even. This simple factoring procedure is best shown in practice in the next example.

\begin{exm} \[\sqrt{741}=[27,\overline{4,1,1,13,18,13,1,1,4,54}]=\sqrt{[27,4,1,1,13]\cdot [27,4,1,1,13,9]}=\]
	
	\[=\sqrt{\dfrac{3321}{122}\cdot \dfrac{30134}{1107}}=\sqrt{\dfrac{3321}{1107}\cdot \dfrac{30134}{122}}=\sqrt{3\cdot 247}\implies 741=3\cdot 247.\]
\end{exm}

Of course, the fact that one can obtain a non-trivial factorization of a natural number from the continued fraction of its square root or a solution of the corresponding Pell equation is not new. This observation already appears in the work of van der Poorten and Walsh~\cite[p. 52, Thm. 1]{van1999note}, where they also mention a connection to the \textit{Lagrange equation}. Furthermore, using the theory of quadratic forms it is possible to obtain faster factorization algorithms, as can be seen in the work of Shanks~\cite{shanks1975analysis}.

However, the advantage of our approach is that the \textit{geometric mean} formula in Theorem~\ref{gmcf} is valid under more general conditions and immediately shows that there is a factorization associated with the continued fraction of the square root.

If the period length of the continued fraction of $\sqrt{n}$ is odd, then it was already known by Legendre~\cite[p. 59-60]{legendre1808essai} that $n$ can be written as a primitive sum of squares. This fact can also be proved using the \textit{geometric mean} formula in Theorem~\ref{gmcfcomplex}.

It is also possible, using the Theorems~\ref{gmcf} and~\ref{gmcfcomplex}, to provide conceptually simpler proofs of some other results about the Pell and Lagrange equations and continued fractions of square roots. In the next section, we illustrate this by analyzing the continued fraction of the square root of a prime odd power.

\section{Continued Fraction of $\sqrt{p^{2m-1}}$}\label{mordell}

In this section, we study, using the procedure of Section~\ref{pell-factor}, the continued fraction of $\sqrt{p^{2m-1}}$, where $p$ is a prime congruent to $3$ modulo $4$ and $m$ is a natural number. As a consequence, we also obtain, following Chakraborty and Saikia~\cite{chakraborty2019conjecture}, a restatement of a conjecture by Mordell~\cite{mordell1961pellian}.

First, note that since $p$ is congruent to $3$ modulo $4$, the period length of the continued fraction of $\sqrt{p^{2m-1}}$ is even. Write $\sqrt{p^{2m-1}}=[a_0/2,\overline{a_1,\dots,a_k,a_{k+1},a_k,\dots,a_1,a_0}]$, where $a_0=2\lfloor\sqrt{p^{2m-1}}\rfloor$ is even and the period is minimal. 

Note that $a_{k+1}$ is odd, because otherwise, by the procedure of Section~\ref{pell-factor}, one can obtain a non-trivial factorization of $p^{2m-1}$ as the product of two coprime natural numbers, which is impossible.

Define, $s_{0,k+1}:=2\cdot K[a_0/2,a_1,\dots,a_k,a_{k+1}/2]$, $s_0:=K[a_0/2,a_1,\dots,a_k]$, $s_{k+1}:=2\cdot K[a_1,\dots,a_k,a_{k+1}/2]$ and $s:=K[a_1,\dots,a_k]$, and observe that $s_{0,k+1},s_0,s_{k+1}$ and $s$ are natural numbers. By the \textit{geometric mean} formula in Theorem~\ref{gmcf} and item $(d)$ of Lemma~\ref{propcont},

\[p^{2m-1}=\dfrac{s_0\cdot s_{0,k+1}}{s\cdot s_{k+1}}=\dfrac{s_0}{s_{k+1}}\cdot \dfrac{s_{0,k+1}}{s},\quad\quad s_0\cdot s_{k+1}-s\cdot s_{0,k+1}=2(-1)^{k+1}.\]

We will prove that $\dfrac{s_0}{s_{k+1}}$ and $\dfrac{s_{0,k+1}}{s}$ are coprime natural numbers and $s_{0,k+1}$, $s_0$, $s_{k+1}$ and $s$ are all odd, and therefore $s_0=s_{k+1}$ and $\dfrac{s_{0,k+1}}{s}=p^{2m-1}$.

Observe that by the second equation above, both $\gcd(s_{0,k+1},s_{k+1})$ and $\gcd(s_{0,k+1},s_{0})$ are equal to $1$ or $2$. It is also clear that $\gcd(s_0,s)=1$. Now, using the definition of continuant we also have, $s_{0,k+1}=a_{k+1}\cdot s_0+2\cdot K[a_0/2,a_1,\dots,a_{k-1}]$ and $s_{k+1}=a_{k+1}\cdot s+2\cdot K[a_1,\dots,a_{k-1}]$, from which follows, since $a_{k+1}$ is odd, that both $s_{0,k+1}$ and $s_0$, and $s_{k+1}$ and $s$ are pairs of natural numbers with the same parity. 

As $\gcd(s_0,s)=1$, we obtain that the parity of the pair $s_{0,k+1}$ and $s_0$ is different from the parity of the pair $s_{k+1}$ and $s$. It follows that $\gcd(s_{0,k+1},s_{k+1})=1$. As a consequence, since $\gcd(s_{0,k+1},s_{k+1})$ and $\gcd(s_{0,k+1},s_{k+1})$ are equal to $1$ and the product $\dfrac{s_0}{s_{k+1}}\cdot \dfrac{s_{0,k+1}}{s}$ is a natural number, we obtain that $\dfrac{s_0}{s_{k+1}}$ and $\dfrac{s_{0,k+1}}{s}$ are natural numbers and $s_0$ and $s_{k+1}$ are odd.

Now, as $\gcd(s_{0,k+1},s_0)$ is either $1$ or $2$, $s_{0,k+1}$ and $s_0$ have the same parity and the product $\dfrac{s_0}{s_{k+1}}\cdot \dfrac{s_{0,k+1}}{s}$ is odd, it then follows that $\dfrac{s_0}{s_{k+1}}$ and $\dfrac{s_{0,k+1}}{s}$ are coprime natural numbers and $s_{0,k+1},s_0,s_{k+1}$ and $s$ are all odd. 

Finally, as $s\leq s_0,s_{k+1}<s_{0,k+1}$, we have that $\dfrac{s_0}{s_{k+1}}<\dfrac{s_{0,k+1}}{s}$ and therefore $\dfrac{s_0}{s_{k+1}}=1$ and $\dfrac{s_{0,k+1}}{s}=p^{2m-1}$, as we wanted to prove.

Using this last fact we can also obtain information about the continued fraction of $\dfrac{1+\sqrt{p^{2m-1}}}{2}$ and $a_{k+1}$. By the \text{geometric mean} formula in Theorem~\ref{gmcf},

\[\sqrt{p^{2m-1}}=\sqrt{\dfrac{s_{k+1}}{s_0}\cdot \dfrac{s_{0,k+1}}{s}}=\sqrt{\dfrac{s_{k+1}}{s}\cdot \dfrac{s_{0,k+1}}{s_0}}=\]

\[=2\sqrt{[a_{k+1}/2,a_k,\dots,a_1]\cdot [a_{k+1}/2,a_k,\dots,a_1,a_0/2]}=\]

\[=2\cdot [a_{k+1}/2,\overline{a_k,\dots,a_1,a_0,a_1,\dots,a_k,a_{k+1}}]\implies\]

\[\dfrac{1+\sqrt{p^{2m-1}}}{2}=[(1+a_{k+1})/2,\overline{a_k,\dots,a_1,a_0,a_1,\dots,a_k,a_{k+1}}],\]

\noindent from which follows that $a_{k+1}$ is either $\lfloor\sqrt{p^{2m-1}}\rfloor$ or $\lfloor\sqrt{p^{2m-1}}\rfloor-1$, whichever is odd.

Mordell's conjecture~\cite[p. 283]{mordell1961pellian} concerns a divisibility property of the fundamental solution of the Pell equation of $p$, where $p$ is a prime congruent to $3$ modulo $4$. This conjecture was inspired by a similar conjecture of Ankeny, Artin and Chowla~\cite[p. 480]{ankeny1952class}.

\begin{con}[Mordell~\cite{mordell1961pellian}]\label{Mordell} Let $(x_0, y_0)$ be the fundamental solution of the Pell equation $x^2-py^2=1$, where $p$ is a prime congruent to $3$ modulo $4$. Then $p$ does not divide $y_0$.
\end{con}

This conjecture can be rewritten in terms of continued fractions. Write $\sqrt{p}=[a_0/2,\overline{a_1,\dots,a_k,a_{k+1},a_k,\dots,a_1,a_0}]$, then Conjecture~\ref{Mordell} is easily seen as equivalent to the statement that $p$ does not divide $K[a_1,\dots,a_k,a_{k+1},a_k,\dots,a_1]$.

In their work, Chakraborty and Saikia~\cite[p. 2551, Thm. 4.1]{chakraborty2019conjecture} proved that Conjecture~\ref{Mordell} is equivalent to the statement that $p$ does not divide $K[a_1,\dots,a_k]$. This is an immediate consequence of our results presented above.

Indeed, note that, by item $(a)$ of Lemma~\ref{propcontsym}, $K[a_1,\dots,a_k,a_{k+1},a_k,\dots,a_1]$ equals $s_{k+1}\cdot s$, and that $s$ is, by definition, equal to $K[a_1,\dots,a_k]$. Therefore, to obtain the equivalence, it is enough to verify that $s_{k+1}$ is not divisible by $p$. But, as we showed, $p=\dfrac{s_{0,k+1}}{s}$ and $\gcd(s_{0,k+1},s_{k+1})=1$, and therefore $s_{k+1}$ is not divisible by $p$.

%%%%%%%%%%%%%%%%%%%%%%%%%%%%%%%%%%%%%%%%%%%%%%%%%%%%%%%%%%%%%%%%%%%%%%%%%%%%%%%%

\section*{Acknowledgements}

This work is partially based on my Ph.D. Thesis at IMPA, Brazil. The author acknowledges the support from CAPES-Brazil scholarship grant.

%%%%%%%%%%%%%%%%%%%%%%%%%%%%%%%%%%%%%%%%%%%%%%%%%%%%%%%%%%%%%%%%%%%%%%%%%%%%%%%%

\IfFileExists{references.bib}
  {\bibliography{references}}
  {\bibliography{../references}}

%%%%%%%%%%%%%%%%%%%%%%%%%%%%%%%%%%%%%%%%%%%%%%%%%%%%%%%%%%%%%%%%%%%%%%%%%%%%%%%%

\end{document}